\documentclass[12pt]{article}
\usepackage{amsmath}
\usepackage{amssymb}
\usepackage[english]{babel}
\usepackage{amsthm}
\usepackage[utf8]{inputenc}

\def \non-ort {{\not\!\!\bot}}

\def \Z  {{\mathbb Z}}
\def \N {{\mathbb N}}

\def \S#1 {{\rm S}#1 }

\def\proclaim #1. #2\par{\medbreak
  \noindent{\bf#1 \enspace}{\sl#2}\par
  \ifdim\lastskip<\medskipamount \removelastskip\penalty55\medskip\fi}

\theoremstyle{plain}
\newtheorem{theorem}{Theorem}
\newtheorem{proposition}[theorem]{Proposition}

\newtheorem{corollary}[theorem]{Corollary}

\theoremstyle{definition}
\newtheorem{definition}[theorem]{Definition}
\newtheorem{example}[theorem]{Example}

\theoremstyle{plain}
\newtheorem{remark}[theorem]{Remark}

\title{On the number of independent orders}
\author{Kota Takeuchi\thanks{Institute of Mathematics, University of Tsukuba.  Partially supported by KAKENHI 19K20209i} \ \ and \ 
Akito Tsuboi\thanks{Institute of Mathematics, University of Tsukuba.  Partially supported by KAKENHI 17K05342.}}
\date{}

\begin{document}
\maketitle

\begin{abstract}
We investigate a model theoretic invariant $\kappa_{srd}^m(T)$, which was introduced by Shelah\cite{shelah}, and prove that $\kappa_{srd}^m(T)$ is sub-additive. 
 When $\kappa_{srd}^m(T)$ is infinite, this gives the equality $\kappa^m_{srd}(T)=\kappa^1_{srd}(T)$, answering a question %by Shelah
 in \cite{shelah}. We apply the same proof method to analyze another invariant $\kappa^m_{ird}(T)$, and show that it is also sub-additive, improving a result in \cite{shelah}. 
\end{abstract}

\section{Introduction}
It is a basic fact that if a theory $T$ is unstable then we can find an unstable 1-formula $\varphi(x,y)$ that witnesses the instability of $T$. (Recall that a formula $\varphi(x,y)$ is called a 1-formula if the length $|x|$ of $x$ is 1.) 
Similar situations are true for some other properties of theories, such as  $TP$, $TP_1$, $TP_2$, $IP$, $IP_n$ and $SOP$. 
Namely, if a theory $T$ has one of these properties, then we can find a $1$-formula 
witnessing the property. 
So, it is of interest to know whether such a $ 1 $-formula exists as a witness for other important  properties of $ T $.
The present paper deals with this kind of question, and we are concerned with the number of independent definable orders existing in the monster model ${\cal M}$ of 
$T$. 

\medbreak
Shelah \cite{shelah} defined three invariants $\kappa_{inp}^m(T)$, $\kappa_{srd}^m (T)$ and $\kappa_{ird}^m (T)$, where $m$ is a positive integer. 
The first, second, and third invariants are concerning the number of independent partitions, independent orders, and independent strict orders existing in $ {\cal M} ^ m $, respectively.
In \cite{shelah}, it was shown that $\kappa_{ird}^m(T)$ does not change its value as $m$ varies  (at least if it is an infinite regular cardinal).
Then it was asked if a corresponding result holds for $\kappa_{inp}^m(T)$ and $\kappa_{srd}(T)$ (\cite[Questions 7.5 and 7.9]{shelah}).
The question about $\kappa_{inp}^m(T)$ was solved in \cite{chernikov2}.
Although the terminology is different, 
Chernikov essentially proved the inequality $\kappa_{inp}^{n+m}(T)\leq  \kappa_{inp}^{n}(T)\times\kappa_{inp}^{m}(T)$, which yields $\kappa_{inp}^m(T)=\kappa_{inp}^1(T)$ if 
$\kappa_{inp}^m(T)$ is infinite. 
Furthermore, he conjectured that the invariant is sub-additive, i.e. $\kappa_{inp}^{n+m}(T)  +1 \leq  \kappa_{inp}^{n}(T)+\kappa_{inp}^{m}(T)$.
This conjecture arose in connection with \cite{kaplan}, in which it was shown that the dp-rank is sub-additive. 
It is known that dp-rank coincides with the rank counting the number of  independent partitions under the assumption of NIP.
Several other invariants (e.g. $\kappa_{cdt}, \kappa_{sct}$) introduced in \cite{shelah} were studied in \cite{chernikov1}, and similar type of results were obtained.

\medbreak
It seems however that $\kappa_{srd}^m(T)$ has not been studied well, and there seems to be no answer to Shelah's question on $\kappa_{srd}^m(T)$. 
Since it has been shown that if $T$ is NIP then $\kappa_{ird}^m(T)=\kappa_{srd}^m(T)$, 
there is no difference between those two invariants under the assumption of NIP.
Under the assumption of NIP,  in \cite{lynn}, the condition $\kappa_{ird}^m(T)<n$ was characterized by using the notion of collapse of indiscernible sequences. 
In this paper we examine how the value $\kappa^m_{srd}(T)$ changes as $m$ changes without any assumption on $T$ (such as NIP). 
We will prove that $\kappa^m_{srd}(T)$ is sub-additive, which gives a positive answer to a question by Shelah. 
The concept of mutually indiscernible sequences plays a central role in our proof technique.
We will also see that the same technique can be applied when analyzing $\kappa^m_{ird}(T)$, and will prove that the invariant is also sub-additive. 
This gives  an improvement of a result in \cite{shelah} on $\kappa^m_{ird}(T)$, when 
it is finite.
\medbreak
Now, we explain some details of $\kappa_{srd}^m(T)$. 
A complete theory  $T$ is said to have the strict order property if  there is a formula $\varphi(x_0,\dots, x_{m-1},y_0, \dots, y_{n-1})$ and parameters $b_i \in {\cal M}^n$ $(i \in \omega)$ 
such that $\Phi:=\{\varphi({\cal M},b_i): i \in \omega\}$ becomes a strictly increasing sequence of 
uniformly defined definable sets of 
${\cal M}^m$, where $\varphi({\cal M},b_i)=\{a \in {\cal M}^m: {\cal M} \models \varphi(a,b_i)\}$. 
Let $\Phi_0=\{D_i: i \in \omega\}$ and $\Phi_1=\{E_i: i \in \omega\}$ be two such strictly increasing 
sequences consisting of subsets of ${\cal M}^n$. 
We say $\Phi_0$ and $\Phi_1$ are independent if 
$(D_{i+1} \setminus D_i) \cap (E_{j+1} \setminus E_j) \neq \emptyset$, for any $(i,j) \in \omega^2$.
We can naturally define the independence among a larger number of  $\Phi_i$'s. 
Then $\kappa_{\mathrm{srd}}^n(T)$ is defined as the minimum cardinal $\kappa$ 
for which there is no family  $\{\Phi_i: i< \kappa\}$ of such independent sequences. 
(See Definition \ref{def1}, for a more precise definition.) 
We put $\kappa_{\mathrm{srd}}(T)=\sup_{n \in \omega} \kappa_{\mathrm{srd}}^n(T)$. 
%
%\medbreak
If there is a (non-trivial) definable order $<$ on ${\cal M}$, then clearly $T$ has the strict order property, and 
$ \kappa_{\mathrm{srd}}^n(T) \geq n+1$.  
Indeed, for an increasing sequence $a_0<a_1< \dots  \in {\cal M}$, 
if we let $X_{i,j}=\{(b_0,\dots,b_{n-1}) \in {\cal M}^n: b_i <a_j\}$, then $\Psi_i=\{X_{i,j}: j \in \omega\}$ $(i<n)$ 
will witness $ \kappa_{\mathrm{srd}}^n(T) \geq n+1$. 

\medbreak
We investigate the invariants $\kappa_{\mathrm{srd}}^n(T)$ and $\kappa_{\mathrm{ird}}^m(T)$,  and prove the following:

\medbreak\noindent
{\bf Theorem A.}
$\kappa_{\mathrm{srd}}^{m+n}(T) +1  \leq  \kappa_{\mathrm{srd}}^m(T)+\kappa_{\mathrm{srd}}^n(T)$. 
\medbreak\noindent
{\bf Theorem B.}
Suppose $\kappa_{\mathrm{srd}}^m(T) \geq \omega$. 
Then $\kappa_{\mathrm{srd}}(T)=\kappa_{\mathrm{srd}}^1(T)=\kappa_{\mathrm{srd}}^m(T)$. 
\medbreak\noindent
{\bf Theorem C.}
$\kappa_{\mathrm{ird}}^{m+n}(T) +1 \leq  \kappa_{\mathrm{ird}}^m(T)+\kappa_{\mathrm{ird}}^n(T)$. 
\medbreak\noindent

\section{Preliminaries}
Let $L$ be a language and $T$ a 
 complete $L$-theory with an infinite model. 
We work in a monster model ${\cal M} \models T$ with a very big saturation. For a set $A \subset {\cal M}$, $L(A)$ denotes the language obtained from $L$ by augmented by the constants for elements in $A$. 
Finite tuples in ${\cal M}$ are denoted by $a,b, \dots$ .
The letters $x,y,\dots$ are used to denote finite tuples of variables. The length  of $x$ 
is denoted by $|x|$. 
Formulas are denoted by $\varphi,\psi,\dots$.  For a formula $\varphi$ and a condition $(*)$, we write $\varphi^{\text{if $(*)$}}$ to denote the formula $\varphi$ if $(*)$ is true,   and 
$\neg \varphi$ if $(*)$ is false. In this paper, we are mainly interested in formulas of the form $\varphi(x,b)$, where $b$ is a parameter from ${\cal M}$. If $|x|=m$, this formula $\varphi(x,b)$ (or $\varphi(x,y)$) will be called an $m$-formula. 
The definable set defined by $\varphi(x,b)$ in ${\cal M}$ is denoted by $\varphi({\cal M},b)$. 

\medbreak
Standard set-theoretic notation will be used. 
\begin{definition}\label{pattern}
Let $\kappa$ be a (finite or infinite) cardinal. 
Let  $(\varphi_i(x;y_i))_{i\in\kappa}$ be a sequence of formulas, and  $(b_{i,j})_{i\in \kappa, j\in \omega}$ a sequence of tuples, 
where $|b_{i,j}|=|y_i|$ for all $i,j$. 
\begin{enumerate}
\item
The pair $\langle (\varphi_i(x;y_i))_{i\in\kappa}, (b_{i,j})_{i\in \kappa, j\in \omega} \rangle$ will be called an ird-pattern of 
width $\kappa$, if it satisfies: 
\begin{enumerate}
\item for any $\eta \in \omega^\kappa$, $\{\varphi_i(x,b_{i,j})^{\text{if $(j \geq \eta(i))$}}: i \in \kappa, j\in \omega\}$ is consistent. 
\end{enumerate}

\item
The pair $\langle (\varphi_i(x;y_i))_{i\in\kappa}, (b_{i,j})_{i\in \kappa, j\in \omega} \rangle$ will be called an srd-pattern of 
width $\kappa$, if it satisfies: 
\begin{enumerate}
\item
for any $\eta\in \omega^\kappa$, $\{\varphi_i(x; b_{i,j})^{\mathrm{if}(j\geq \eta(i))}: i\in \kappa, j\in\omega\}$ is consistent,
\item
for each $i\in\kappa$ and $j\in\omega$, $\varphi({\cal M}, b_{i,j}) \subsetneq \varphi({\cal M}, b_{i,j+i})$.
\end{enumerate}
\end{enumerate}
\end{definition}

\begin{definition}\label{def1} 
Let $ * \in\{\text{ird, srd}\}$.
$\kappa^m_{*}(T)$ is the minimum cardinal $\kappa$ such that (in $T$) there is no $*$-pattern of width $\kappa$ witnessed by $m$-formulas $\varphi_i(x;y_i)$ $(i\in\kappa)$.
We write $\kappa^m_{*}(T)=\infty$, if there is no such $\kappa$. 
Also $\kappa_{*}(T)$ is defined as $\sup_{m\in\omega}\kappa^m_{*}(T)$.
\end{definition}

\begin{remark}
\begin{enumerate}
\item
$\kappa^m_{\mathrm{ird}}(T)>1$ if and only if there is an unstable formula $\varphi(x,y)$ with $|x|=m$.
$\kappa_{\mathrm{ird}}(T)>1$ if and only if $T$ is unstable.
\item
$\kappa^m_{\mathrm{srd}}(T)>1$ if and only if there is a $\varphi(x,y)$ with $|x|=m$ having the strict order property. $\kappa_{\mathrm{srd}}(T)>1$ if and only if $T$ has the strict order property.
\end{enumerate}
\end{remark}
If $\kappa < \kappa^m_{srd}(T)$, then there are $\kappa$-many $\varphi_i$'s and a set $B=(b_{i,j})_{i \in \kappa, j \in \omega}$ satisfying the conditions (a) and (b) of the item 2 in Definition \ref{pattern}. 
The condition (b) states that each $\varphi_i$ defines a strict order on ${\cal M}^m$, and the condition (a) states that 
the orders defined by $\varphi_i$'s are independent. 
If $\kappa_{\mathrm{srd}}^m(T)=\infty$, then there is a set $\{\varphi_i(x,y_i):i<|T|^+\}$  witnessing the conditions.  
So, by choosing an infinite subset of $|T|^+$, we can assume $\varphi_i=\varphi$ for all $i<\omega$. 
Conversely, if $\kappa_{srd}(T) \geq \omega$ and if the witnessing formulas satisfy $\varphi_i = \varphi$ $(i<\omega)$, then by compactness, we see that there are arbitrarily many independent strict orders. 
Notice also that if $\kappa^m_{srd}(T)=\infty$ then $T$ has the independence property.  
\begin{example}
Let $T$ be the theory of $\N=(\N,0,1,+,\cdot)$. 
Let $\varphi(x,y_0,y_1)$ be the formula asserting that the exponent of the $y_0$-th prime in the prime factorization of $x$ is smaller than $y_1$. 
Then, for each $i$, $\Phi_i:=\{\varphi({\cal M},i,j)\}_j$ forms an increasing sequence of definable sets. Moreover, $\Phi_i$'s  are independent, 
so we have $\kappa^1_{srd}(T)=\infty$. 
\end{example}

Indiscernibility is a substantial concept in modern model theory. In our paper \cite{kota},  a couple of results concerning the existence of an indiscernible tree are presented. Here in this paper, the notion of mutual indiscernibility is important. 
\begin{definition}
A set $\{B_i: i<\kappa\}$ of indiscernible sequences is said to be mutually indiscernible over $A$ if for every $i<\kappa$, the sequence $B_i$ is indiscernible over $A \cup \bigcup_{i\neq j<\kappa} B_j$.
\end{definition}
The following proposition is simple to prove, but plays an important role in our argument.
\begin{proposition}\label{prop}
For each $i<\kappa$, let $B_i=(b_{i,j})_{j\in\omega}$ be an infinite sequence of tuples of the same length. 
Let $\Gamma((X_i)_{i<\kappa})$ be a set of formulas, where $X_i=(x_{i,j})_{j \in \omega}$ $(i < \kappa)$ and $|x_{i,j}|=|b_{i,j}|$. 
We assume the following property for $\Gamma$:
\begin{itemize}
\item[(*)] 
if $B'_i $ is an infinite subsequence of $B_i$ $(i<\kappa)$ then $(B'_i)_{i<\kappa}$ realizes  $\Gamma((X_i)_{i<\kappa})$. 
\end{itemize}
Then, for any set $A$, we can find $\{C_i:i<\kappa\}\models \Gamma((X_i)_{i<\kappa})$ that is mutually indiscernible over $A$.
\end{proposition}

The following observation, shown by Proposition \ref{prop}, is a key in our proof of Theorem \ref{main}. 

\begin{remark}\label{indiscernible witness}\rm
Let $Z$ denote $\Z$ or $\Z\cup\{\pm \infty\}$.
Then, there is an srd-pattern of width $\kappa$ witnessed by a sequence $(\varphi_i(x;y_i))_{i\in\kappa}$ of formulas  if and only if
there are tuples $a$ and $b_{i,j}$ ($i\in\kappa, j\in Z$) with the following properties:
\begin{enumerate}
\item
For all $i\in\kappa$ and $j \leq k \in Z$, $\varphi_i({\cal M}, b_{i,j}) \subset \varphi_i({\cal M}, b_{i,k})$;
\item
$\{B_i:i\in \kappa\}$ is mutually indiscernible, where $B_i=(b_{i,j})_{j\in Z}$;
\item
For all $i\in\kappa$ and $j \in Z$, ${\cal M} \models\varphi_i(a,b_{i,j})$ if and only if $j \geq 0$.
\end{enumerate}
In the equivalence above, we can also assume the following condition in addition to 1 --3.

\begin{enumerate} \setcounter{enumi}3
\item $\{B_{i,+} :i\in\kappa\} \cup \{B_{i,-} :i\in\kappa\}$ is mutually indiscernible over $a$, i.e. $B_{i,+}$ is indiscernible over $\{a\} \cup B_{i,-} \cup \bigcup_{i'\neq i}B_{i'}$ and $B_{i,-}$ is indiscernible over $\{a\} \cup B_{i,+} \cup \bigcup_{i'\neq i}B_{i'}$, where $B_{i,+}=(b_{i,j})_{j\geq 0}$ and $B_{i,-}=(b_{i,j})_{j< 0}$.
\end{enumerate}
\end{remark}

\begin{remark} \label{second}
Let $(D_i)_{i \in I}$ be an increasing sequence of sets in ${\cal M}^n$, where $I$ is a linearly ordered set. 
Then the following sequences are also increasing:
\begin{enumerate}
\item $(D_i \cap D)_{i \in I}$, where $D$ is a subset of ${\cal M}^n$;
\item $(\pi (D_i))_{i \in I}$, where $\pi:{\cal M}^n \to {\cal M}^m$ is the projection $(x_0, \dots, x_{n-1}) \mapsto (x_{i_0}, \dots, x_{i_{m-1}})$. 
\end{enumerate}
\end{remark}

\section{Main Results}

In the following theorem, $\kappa, \kappa_0$ and $\kappa_1$ are arbitrary cardinals, but the interesting case is when 
they are finite.

\begin{theorem}\label{main}
Let $\kappa, \kappa_0$ and $\kappa_1$ be cardinals such that $\kappa+1=\kappa_0 + \kappa_1$.
Suppose that there is an srd-pattern of width $\kappa$ with formulas $\varphi_i(x;y_i)$ $(i\in\kappa)$, where $x=x_0x_1$.
Then, there is $l\in\{0,1\}$ for which we can find formulas $\psi_i(x_l; y'_i)$ $(i\in\kappa_l)$ witnessing the definition of srd-pattern of width $\kappa_l$. 
\end{theorem}

\begin{proof}
Let $Z=\Z \cup \{\pm \infty\}$ and 
choose $b_{i,j}$ $(i\in\kappa, j \in Z)$ and $a$ satisfying the conditions 1 -- 4 in Remark \ref{indiscernible witness}. 
We  write $a$ in the form $a=a_0a_1$, where $|a_0|=|x_0|$ and $|a_1|=|x_1|$.
For $\eta \in \Z^{\kappa}$, let 
\[
\Delta_\eta(x_0, a_1):=\{\varphi_i(x_0, a_1, b_{i,j})^{\text{if }j\geq \eta(i)}: i\in \kappa, j\in Z\}. 
\]
Then choose a maximal $F\subset \kappa$ satisfying the following property:
\begin{itemize}
\item[(*)] For any $\eta\in \Z^{\kappa}$ with $\mathrm{supp}(\eta) \subset F$ (i.e., $\eta(i)=0$ if $i \notin F$), 
$\Delta_\eta(x_0, a_1)$ is consistent. 
\end{itemize}

There are two complementary cases: 

\medbreak\noindent
{\bf Case 1: } 
Suppose $|F|\geq \kappa_0$.
In this case the proof is straightforward for $l=0$, since the formulas 
$\varphi_i(x_0; x_1y_i)$ $(i\in F)$ and the tuples 
$c_{i,j}=a_1b_{i,j}$ $(i \in F, j\in\omega)$ form an srd-pattern of width $\kappa_0$e.

\medbreak\noindent
{\bf Case 2: } 
Suppose $|F|< \kappa_0$. 
Then the set $\kappa \setminus F$ has the cardinality $\geq \kappa_1$. 
Without loss of generality, we can assume $\kappa_1\subset \kappa\setminus F$.
In this case, for any $\alpha\in\kappa_1$, the extension $F\cup\{\alpha\}\supset F$ does not satisfy $(*)$. 
Namely, there is $\eta$ with $\mathrm{supp}(\eta) \subset F \cup \{\alpha\}$, for which the set $\Delta_\eta(x_0,a_1)$ is inconsistent.  
Fix $\alpha\in \kappa_1$ for a while. 
Since $\{\varphi_i({\cal M}, a_1,b_{i,j}):j\in Z\}$ is a strictly increasing sequence for each $i$, we can 
choose $\eta_0\in \Z^F$ and $m \in Z\setminus\{0\}$ such that the subset 
\begin{equation*}
\begin{split}
&\{\varphi_i(x_0, a_1, b_{i,\eta_0(i)}), \neg\varphi_i(x_0, a_1, b_{i,\eta_0(i)-1}): i\in F\}\\
&\quad\cup
\{\varphi_{\alpha}(x_0, a_1, b_{\alpha,m}), \neg\varphi_{\alpha}(x_0, a_1, b_{\alpha,m-1})\}\\
&\quad\cup
\{\varphi_{i}(x_0, a_1, b_{i,0}), \neg\varphi_{i}(x_0, a_1, b_{i,-1}): i\in \kappa\setminus(F \cup\{\alpha\})\}
\end{split}
\end{equation*}
of $\Delta_\eta$ is inconsistent. 
Since the other case is similar and in fact easier, we assume $m > 0$.
Then, by compactness, and since 
$\{B_{i,+} :i\in\kappa\} \cup \{B_{i,-} :i\in\kappa\}$ is mutually indiscernible over $a$, 
we can find finite sets $F_0 \subset F$ and $F_1 \subset \kappa\setminus(F \cup\{\alpha\})$ such that 
\begin{equation*}
\begin{split}
\Sigma_\alpha(x_0) &:=\{\varphi_i(x_0, a_1, b_{i,\eta_0(i)}), \neg\varphi_i(x_0, a_1, b_{i,\eta_0(i)-1}): i\in F_0\}\\
&\quad\cup \{\varphi_{\alpha}(x_0, a_1, b_{\alpha,\infty}), \neg\varphi_\alpha(x_0,a_1, b_{\alpha, 0})\}\\
&\quad\cup \{\varphi_i(x_0, a_1, b_{i,\infty}), \neg\varphi_i(x_0,a_1, b_{i, -\infty}): i\in F_1\} %
\end{split}
\end{equation*}
Now, let \[
B^*:= \{b_{i,j}\}_{ i\in F, j\in \Z}\, \cup \, \{b_{i, -\infty} \}_{i\in \kappa\setminus F} \, \cup \, \{b_{i, \infty} \}_{i\in \kappa\setminus F} %
\]
Then the parameters appearing in $\Sigma_\alpha(x_0)$, other than $B^*$, are  $a_1$ and $b_{\alpha, 0}$.  
(The definition of $B^*$ does not depend on $\alpha$ and hereafter we work with the language  $L(B^*)$.) 
So we write $\Sigma_\alpha$ as $\Sigma_\alpha(x_0,a_1,b_{\alpha,0})$.  By preparing a variable $z_\alpha$ with $|z_\alpha|=|b_{\alpha,j}|$,  let $\psi_\alpha'(x_0,x_1,z_\alpha)$ be the formula $\bigwedge \Sigma_\alpha(x_0,x_1,z_\alpha)$.  
Recall that the set $\Sigma_\alpha(x_0,a_1,b_{\alpha,0})$ is inconsistent. 
However, the set $\Sigma_\alpha(x_0,a_1,b_{\alpha,-1})$ is consistent, by our choice of $F$ and the condition $(*)$. 
By the condition 4 in Remark \ref{indiscernible witness}, this means that $\psi'_\alpha(x_0,a_1,b_{\alpha,j})$ is consistent if and only if $j < 0$. 
So, if we define 
\begin{eqnarray*}
\psi_\alpha(x_1, z_\alpha)&:=& (\exists x_0) \, \psi'_\alpha(x_0,x_1,z_\alpha),\\
c_{\alpha,j}&:=&b_{\alpha,-j-1},
\end{eqnarray*}
then we have 
\[
{\cal M} \models \psi_\alpha(a_1, c_{\alpha,j}) \iff j \geq 0. 
\]
Since this is true for all $\alpha \in \kappa_1$, it follows that $\langle (\psi_\alpha)_{\alpha \in \kappa_1}, (c_{\alpha,j})_{\alpha \in \kappa_1,j \in \Z}\rangle$ satisfies the 
condition 3 in Remark \ref{indiscernible witness}.
The condition 2 is easily shown, since the sequences $(c_{\alpha,j})_{j\in \Z}$ $( \alpha\in \kappa_1) $ are mutually indiscernible over $B^*$.
Finally the condition 1 follows from Remark \ref{second}.
Hence, $\langle (\psi_\alpha)_{\alpha \in \kappa_1}, (c_{\alpha,j})_{\alpha \in \kappa_1,j \in \Z}\rangle$ is an srd-pattern of width $\kappa_1$. 
\end{proof}

\begin{corollary}
\begin{enumerate}
\item
$\kappa^{m+n}_{\mathrm{srd}}(T)+1 
\leq \kappa^{m}_{\mathrm{srd}}(T)+\kappa^{n}_{\mathrm{srd}}(T)$.
\item
If $\kappa^m_{\mathrm{srd}}(T)$ is infinite, then $\kappa^m_{\mathrm{srd}}(T)=\kappa^1_{\mathrm{srd}}(T)=\kappa_{\mathrm{srd}}(T)$.
\end{enumerate}
\end{corollary}

\begin{proof}
We only prove the first item.  We can assume $\kappa^{m+n}_{\mathrm{srd}}(T)$ is finite, since the infinite case is easier.
By way of a contradiction, we assume $\kappa^{m+n}_{\mathrm{srd}}(T) +1>  \kappa^{m}_{\mathrm{srd}}(T)+\kappa^{n}_{\mathrm{srd}}(T)$. 
Then there must be an srd-pattern of width $\kappa:= \kappa^{m}_{\mathrm{srd}}(T)+\kappa^{n}_{\mathrm{srd}}(T) -1$ witnessed by $(m+n)$-formulas. 
By Theorem \ref{main}, using the equation 
$\kappa+1=\kappa^{m}_{\mathrm{srd}}(T)+\kappa^{n}_{\mathrm{srd}}(T)$, 
we would have (i) the existence of an srd-pattern of width $\kappa^{m}_{\mathrm{srd}}(T)$ by $m$-formulas, or (ii) the existence of an srd-pattern of width $\kappa^{n}_{\mathrm{srd}}(T)$ by $n$-formulas.  In either case, we reach a contradiction.
\end{proof}

The above argument can be applied to show the corresponding result for $\kappa_{\mathrm{ird}}^m(T)$. 
The following theorem  on  $\kappa_{\mathrm{ird}}^m(T)$ gives an improvement of  \cite[Theorem 7.10]{shelah}. (In that book he investigated $\kappa^m_{\mathrm{ird}}(T)$ when it is infinite.)  
In the following theorem,  $\kappa, \kappa_0$ and $\kappa_1$ are any cardinals as before.
\begin{theorem}
Assume 
 $\kappa+1=\kappa_0 + \kappa_1$.
Suppose that there is an ird-pattern of width $\kappa$ with formulas $\varphi_i(x;y_i)$ $(i\in\kappa)$, where $x=x_0x_1$.
Then, there is $l\in\{0,1\}$ for which we can find formulas $\psi_i(x_l; y'_i)$ $(i\in\kappa_l)$ witnessing an ird-pattern of width $\kappa_l$. 
\end{theorem}

\begin{proof}
The outline of the proof is quite similar to that of Theorem \ref{main}.
However, for completeness, the details of the proof are provided.
In the present proof, our linear order $Z$ has the form $Z=\Z_- + \Z + \Z_+$, where both $\Z_-$ and $\Z_+$ are copies of $\Z$, and the order is defined 
so that $\Z_- < \Z < \Z_+$.

Choose $b_{i,j}$ $(i\in\kappa, j \in Z)$ and $a=a_0a_1$ satisfying the conditions 1 -- 4 in Remark \ref{indiscernible witness}. 
Then for $\eta \in \Z^{\kappa}$, consider the set $\Delta_\eta(x_0, a_1)$, which is defined in the same way as in the proof of previous theorem. 
Again, choose a maximal $F\subset \kappa$ satisfying the following property:
\begin{itemize}
\item[(**)] For any $\eta\in Z^{\kappa}$ with $\mathrm{supp}(\eta) \subset F$, 
$\Delta_\eta(x_0, a_1)$ is consistent. 
\end{itemize}

\medbreak\noindent
{\bf Case 1: } 
Suppose $|F|\geq \kappa_0$.
The proof is straightforward as the previous theorem so we skip this case.

\medbreak\noindent
{\bf Case 2: } 
Suppose $|F|< \kappa_0$. Without loss of generality, we can assume $\kappa_1\subset \kappa\setminus F$.
In this case, for any $\alpha\in\kappa_1$, 
there is $\eta$ with $\mathrm{supp}(\eta) \subset F \cup \{\alpha\}$, for which the set $\Delta_\eta(x_0,a_1)$ is inconsistent.  
By compactness, we can choose finite sets $F_0 \subset F$, $F_1 \subset \kappa \setminus (F \cup\{ \alpha\})$, and $U_i, O_i \subset \Z$ $(i \in F_0 \cup  F_1 \cup \{\alpha\})$ 
with the following properties:
\begin{enumerate}
\item  $U_i < O_i$, for any $i$;
\item $U_i < 0 \leq O_i$, if $i \in F_1$;
\item The following set $\Sigma_\alpha(x_0)$ is inconsistent: 
\[
\begin{split}
\{ \neg \varphi_i(x_0,a_1,b_{i,j}): i \in F_0, j \in U_i\} & \cup \{\varphi_i(x_0,a_1,b_{i,j}): i \in F_0, j \in O_i\} \\
\cup \{ \neg \varphi_\alpha(x_0,a_1,b_{i,j}):  j \in U_\alpha\} & \cup \{\varphi_\alpha(x_0,a_1,b_{i,j}):  j \in O_\alpha\} \\
\cup \{ \neg \varphi_i(x_0,a_1,b_{i,j}): i  \in F_1, j \in U_i \} & \cup \{\varphi_i(x_0,a_1,b_{i,j}): i \in  F_1, j \in O_i\}. 
\end{split}
\]
\end{enumerate}
If $U_\alpha < 0 \leq O_\alpha$ holds, then $\Sigma_\alpha$ must be consistent, by our choice of $F$. 
So, since the other case is similarly proven, 
we can assume 
 $U_\alpha^+:=\{j \in U_\alpha: 0 \leq j\} \neq \emptyset$.  Moreover $U_\alpha$ is assumed to be chosen so that $|U_\alpha^+|$ is minimum.

Since 
$\{B_{i,+} :i\in\kappa\} \cup \{B_{i,-} :i\in\kappa\}$ is mutually indiscernible over $a$, 
we can assume 
\begin{itemize}
\item $U_i , O_i \subset \Z$ $(i \in F)$;
\item $U_i \subset \Z_-, \ O_i \subset \Z_+$ $ (i \in \kappa \setminus (F \cup \{\alpha\}))$;
\item $U_\alpha^-:=\{j \in U_\alpha:j<0\} \subset \Z^-$, $U_\alpha^+=\{0, \dots,k-2, k-1\} \subset \Z$; 
\item  $O_\alpha \subset \Z_+$.
\end{itemize}
Now, let \[
B^*:= \{b_{i,j}: i\in F, j\in \Z\}\, \cup \, \{b_{i, j} : i\in \kappa\setminus F, j \in \Z_- \cup \Z_+ \} 
\]
Then the parameters appearing in $\Sigma_\alpha(x_0)$, other than $B^*$, are  $a_1$ and $(b_{\alpha, j})_{j \in U_\alpha^+}$. 
So we write $\Sigma_\alpha$ as $\Sigma_\alpha(x_0,a_1,(b_{\alpha,j})_{j \in k})$.  
Let $\psi_\alpha'(x_0,x_1,z_\alpha)$ be the formula $\bigwedge \Sigma_\alpha(x_0,x_1,z_\alpha)$.  
Recall that the set $\Sigma_\alpha(x_0,a_1,(b_{\alpha,j})_{j \in k})$ is inconsistent. 
However, the set $\Sigma(x_0,a_1,(b_{\alpha,j})_{j \in \{-k,\dots , -1\}})$ is consistent, by the choice of $F$. 
By the condition 4, if we set $c_{\alpha,l}=(b_{\alpha,j})_{j \in \{lk, lk+1, \dots, lk+(k-1)\}}$, this means that $\psi'_\alpha(x_0,a_1,c_{\alpha,j})$ is consistent if and only if $j < 0$. 
The rest of the proof is almost identical with that of $srd$-case.
\end{proof}

From this theorem we deduce the following corollary.  The item 2 is essentially shown in [1].
\begin{corollary}
\begin{enumerate}
\item
$\kappa^{m+n}_{\mathrm{ird}}(T) +1 \leq  \kappa^{m}_{\mathrm{ird}}(T)+\kappa^{n}_{\mathrm{ird}}(T)$.
\item 
If $\kappa^m_{\mathrm{ird}}(T)$ is infinite, then $\kappa^m_{\mathrm{ird}}(T)=\kappa^1_{\mathrm{ird}}(T)=\kappa_{\mathrm{ird}}(T)$.
\end{enumerate}
\end{corollary}

\end{document}